\documentclass[10pt]{amsart}
\usepackage{amsthm, amsfonts, amssymb, amsmath, amscd, mathtools}
\usepackage[all]{xy}
\usepackage{lmodern}
\usepackage[margin=3.5cm]{geometry}

\newcommand{\into}{\hookrightarrow}
\newcommand{\onto}{\twoheadrightarrow}

\newcommand{\C}{\mathbb{C}}
\newcommand{\R}{\mathbb{R}}

\newcommand{\Z}{\mathbb{Z}}

\newcommand*{\abs}[1]{\lvert#1\rvert}
  
\newcommand{\F}{\mathbb{F}}
\newcommand{\T}{\mathbb{T}}
\newcommand{\That}{\widehat{\Z}}
\newcommand{\Zhat}{\widehat{\Z}}

\newcommand{\LL}{\mathcal{L}}

\newcommand{\Bound}{\LL}

\newcommand{\K}{\mathrm{K}}

\newcommand{\KK}{\mathrm{KK}}

\newcommand{\PD}{\textup{PD}}

\newcommand{\dol}{\overline{\partial}}

\newcommand{\FM}{\mathcal{F}}

\newcommand{\corl}[1]{\overset{#1}{\longleftarrow}}
\newcommand{\corr}[1]{\overset{#1}{\longrightarrow}}
\newcommand{\hotimes}{\hat{\otimes}}

\newcommand{\PDspin}{\mathrm{PD}_{\mathrm{spin}}}
\newcommand{\PDfm}{\mathrm{PD}_{\textup{B-C}}}
\newcommand{\poincare}{\mathcal{P}}

\newcommand{\Deltafm}{\Delta_{\textup{B-C}}}
\newcommand{\Deltaspin}{\Delta_{\textup{Spin}}}
\newcommand{\Dudeltaspin}{\widehat{\Delta}_{\textup{Spin}}}
\newcommand{\Dudeltafm}{\widehat{\Delta}_{\textup{B-C}}}
\newcommand{\Deltads}{\widehat{\Delta}_{\textup{DS}}}
\newcommand{\derham}{D_{\textup{dR}}}

\newcommand{\pnt}{\cdot}

\newcommand{\pr}{\textup{pr}}

\newcommand{\id}{\mathrm{id}}
\newcommand*{\norm}[1]{\lVert#1\rVert}

\numberwithin{equation}{section}

\usepackage{microtype}
\numberwithin{equation}{section}
\theoremstyle{theorem}
\newtheorem{theorem}[equation]{Theorem}
\newtheorem{lemma}[equation]{Lemma}
\newtheorem{proposition}[equation]{Proposition}
\newtheorem{corollary}[equation]{Corollary}
\theoremstyle{definition}
\newtheorem{definition}[equation]{Definition}

\theoremstyle{remark}
\newtheorem{remark}[equation]{Remark}
\newtheorem{example}[equation]{Example}

\begin{document}
\title{Baum-Connes and the Fourier-Mukai transform}


\author{Heath Emerson}
\email{hemerson@math.uvic.ca}
\address{Department of Mathematics and Statistics\\
  University of Victoria\\
  PO BOX 3045 STN CSC\\
  Victoria, B.C.\\
  Canada V8W 3P4}

\author{Dan Hudson}
\email{dhudson@math.toronto.edu}
\address{Department of Mathematics\\
	University of Toronto\\
	40 St. George Street\\
	Toronto, ON\\
	Canada M5S 2E4}


\date{\today}

\thanks{This research was supported by an NSERC Discovery grant and an NSERC Alexander Graham Bell CGS-M. }

\begin{abstract}

The Fourier-Mukai transform from algebraic geometry may be formulated 
in KK-theory as the map of composition with a certain topological correspondence 
in the sense of Connes and Skandalis. The goal of this note is to analyze this correspondence 
and to describe the induced map in terms of certain natural Baum-Douglas cycles and 
co-cycles for tori. This leads to a purely geometric description of the 
Baum-Connes assembly map for free abelian groups. 

\end{abstract}

\maketitle


\section{Introduction}
A \emph{topological correspondence} between smooth manifolds \(X\) and \(Y\), a concept due to Connes and Skandalis \cite{Connes-Skandalis}, is the content of a diagram 
	\begin{equation}
		\label{equation:correspondence}
			X \xleftarrow{b} (M, \xi) \xrightarrow{f} Y,
	\end{equation}
where \(b\) is a smooth map, \(f\) a \(\K\)-oriented smooth map, and \(\xi\) is a K-theory class for \(M\).  A correspondence determines an element of 
	\[ \KK_* (C_0(X), C_0(Y)),\]
since \(b\) determines an element by ordinary functoriality \(b^*\in \KK_0(C_0(M), C_0(X))\), \(f\colon M \to Y\) a wrong way, or `shriek' morphism \(f! \in \KK_{\dim Y - \dim M}\bigl(C_0(M), C_0(Y)\bigr)\), and the \(\K\)-theory datum can be integrated as a `twist,' using the ring structure on topological K-theory. The references \cite{EM:Embeddings} and \cite{EM:Geometric_KK} develop a geometric module for \(\KK\) for manifolds, using equivalence classes of correspondences. 

The correspondence concept is very closely related to one in algebraic geometry, sometimes called a \emph{Fourier-Mukai transform}. If \(X\) and \(Y\) are smooth projective varieties, then a suitable object \(\mathcal{E}\) in the derived category \(\mathcal{D}^b(X\times Y)\) of sheaves over \(X\times Y\), gives rise to a transformation 
	\begin{equation}
		\mathcal{D}^b (X) \xrightarrow{ (Rp_X)^*} \mathcal{D}^b(X\times Y)
		\xrightarrow{\cdot \otimes^L \mathcal{E}} \mathcal{D}^b(X\times Y)\xrightarrow{ (Rp_Y)_*} \mathcal{D}^b(Y).
	\end{equation}
between derived categories of sheaves, where the maps \( (Rp_X)_*\) the derived inverse image functor and \((Rp_Y)_*\) the derived direct image functor, for the coordinate projections \(p_X, p_Y\). For example, if \(X = Y\), the structure sheaf of the diagonal \(\Delta \subset X\times X\) can be used, and the induced map is the identity. Mukai proved that if \(T\) was an abelian variety, \(\hat{T}\) its dual, and \(\mathcal{E}\) the structure sheaf of sections of the \emph{Poincar\'e bundle} \(\poincare\) then the resulting transform 
	\[\FM \colon \mathcal{D}^b(T) \to \mathcal{D}^b(\hat{T})\]
is an isomorphism, with inverse the map induced by the dual Fourier-Mukai transform obtained by flipping \(T\) and \(\hat{T}\). The isomorphism is called \emph{Fourier-Mukai duality}. It is an instance of a \emph{\(T\)-duality} in physics, relating two string theories with different space-time geometries. 

If \(T^d = \R^d/\Z^d\) is the torus, the \emph{Poincar\'e bundle} \(\poincare_d\) is the complex line bundle over \(T^d\times \That^d\) given by 
	\[ \poincare_d=\R^d\times \That^d \times_{\Z^d}\C\; / \; \sim\]
where \(\sim\) identifies 
	\[(x,\chi,\lambda) \sim (x+n,\chi, \chi(n)\lambda), \quad \text{for $n\in \Z^d$,}\]
and the bundle projection is induced by the coordinate projection \(\R^d\times \That^d\times \C\to \R^d\times \That^d\). 

The \emph{Fourier-Mukai correspondence} is the correspondence 
	\[\T^d \corl{\pr_1} (\T^d\times \That^d, \poincare_d) \corr{\pr_2} \That^d, \]
where \(\poincare_d\) is the class of the Poincar\'e line bundle and \(\pr_2\) has a canonical K-orientation. The Fourier-Mukai correspondence defines a class 
	\[[\FM_d]\in \KK_{-d} (\T^d, \That^d)\]
in geometric KK, and in analytic KK. 

It has been well known to some experts for a while
 that the Fourier-Mukai morphism is closely 
 related to the Baum-Connes map for free abelian groups, although it seems it has 
 not been written down anywhere, until \cite{Emerson:Fibre}. Let
\([D]\) be the class of the \(\Z^d\)-equivariant Dirac operator on \(\R^d\) (the Dirac morphism for 
\(\Z^n\) in the sense of \cite{Meyer-Nest}) ). We apply the following map to it. 
We first apply Kasparov's descent map 
	\[j\colon \KK^{\Z^d}_* ( C_0(\R^d), \C) \to \KK_{-d} \left(C(\T^d), C^*(\Z^d) \right) \cong \KK_{-d} \left( C(\T^d), C(\That^d) \right)\]
and Fourier transform. We then compose with the class of the standard 
Morita equivalence bimodule, in 
\(\KK_0(C(\T^d), C_0(\R^d)\rtimes \Z^d)\). Then the composition of these maps 
sends \([D] \in \KK_{-d}(C_0(\R^d), \C)\) to the class of the Fourier-Mukai correspondence. 
The appearance of the Poincar\'e bundle in the correspondence is due to the 
 the Morita equivalence bimodule, which has a twisting effect.

In this note we describe the
 Baum-Connes assembly map geometrically using the 
Fourier-Mukai correspondence and the various geometric tools available in the correspondence theory, 
and a natural parameterization of the K-homology of tori in terms of classes of oriented  sub-tori. 

An \emph{oriented} subtorus \(T = V/\Gamma\) in \(\T^d\) is a torus subgroup of \(\T^d\) together with a K-orientation on it; this is equivalent to an orientation on \(V\). An oriented subtorus gives a correspondence \( \T^d\xleftarrow{i} T\rightarrow \pnt\) (\(i\) the inclusion) and class \( [(T,i)]_*\in \KK_{-\dim T} (\T^d, \pnt)\). On the other hand, the inclusion \(i\) is a K-oriented embedding so \( \pnt \leftarrow T \xrightarrow{i} \T^d\) gives a K-theory counterpart \([(T,i)]!\in \KK_{\dim T}. (\pnt, \T^d)\). 

If \(i \colon T \to \T^d\) is an oriented subtorus where \( T = V/\Gamma\), \(V\) oriented, then \(\widehat{V}\) has an induced K-orientation and and so does \(\widehat{\Gamma}\). The exact sequence 
	\[0 \corr{ } \ker(\hat{i}) \corr{ } \Zhat^d \corr{\hat{i}} \widehat{\Gamma} \corr{} 0, \]
endows $\ker(\hat{i})$ with a K-orientation. The subtorus \(\hat{T} := \ker (\hat{i}) \subset \That^d\) is called the \emph{dual subtorus}. We prove the following. 

	\begin{theorem}
		Let \(T \subset \T^d\) be a K-oriented subtorus of dimension \(k\) defining the element \([(T,i)]_* \in \K_{j} (\T^d)\), and \(\hat{T}\) the dual torus. Then 
			\[\mu ([(T,i)]_*) =  (-1)^{k(d-k)+\frac{k(k-1)}{2}}\cdot [\widehat{(T,i)}]! \in \K^{-j} (\That^d),\]
		where \(\mu\) is the Baum-Connes assembly map. 
	\end{theorem}
We use the Theorem to deduce the Fourier-Mukai inversion formula in K-theory. 

The nice properties of the Fourier-Mukai transform in \emph{complex} KK-theory, are due to Bott Periodicity. In terms of correspondences, the correspondence 
	\[\T \leftarrow (\pnt, \mathbf{1}) \rightarrow \pnt,\]
is equivalent in the correspondence framework to 
	\[\T \corl{\pr_1} (\T\times \Zhat, \poincare_1) \corr{ } \pnt.\]
The two correspondences are \emph{Bott equivalent} (or \emph{Thom} equivalent). This kind of equivalence is built into the correspondence framework to make it Bott Periodic. In going from the second correspondence to the first, the dimension of the middle space has been reduced by \(2\). The proof of the Theorem amounts to repeated application of the method of simplifying the middle space of a correspondence, using the Poincar\'e bundle and Bott equivalence. 

The Baum-Connes assembly map for \(\Z^d\), building in Fourier-transform \(C^*(\Z^d) \cong C(\That^d)\), factors through spin duality for a torus, and the Fourier-Mukai transform: 
	\[\xymatrix{  \K_*(\T^d)  \ar[r]^{\PDspin \;\;\;\;} \ar[dr]_{\mu} & \K^{*+d} (\T^d) \ar[d]^{(\FM_d)_*} \\  & \K^*(\That^d) }.\]
The assembly map \(\mu\) is analytically defined. It associates to the K-homology class of an oriented subtorus a certain families index problem: the familes index is a K-theory class for \(\That^d\). But the previous theorem implies that \(\mu ([T]) = [\hat{T}]!\in \K^* (\That^d)\), up to a specified sign. 

This gives therefore the promised geometric description of the assembly map for free abelian 
groups. 

\section{Topological Bivariant K-Theory}

The theory of topological correspondences goes back to Connes and Skandalis and as developed further in \cite{EM:Geometric_KK}, in particular, shown to form a category naturally isomorphic to Kasparov's \(\KK\), when the arguments are smooth manifolds. The fact that the topological correspondences (of \cite{EM:Geometric_KK}) map to the analytically defined ones of \cite{Connes-Skandalis} is a very general way of stating the Index Theorem of Atiyah and Singer. 

\subsection{Topological Correspondences}

In this paper we will be operating in the environment of of this topological picture of \(\KK\)-theory and so give a quick summary of it before proceeding to the application. For the purposes of this paper we will always take $X$ to be compact, and adjust the definitions accordingly. For a more in depth account of the theory, see~\cite{EM:Geometric_KK}. 

\begin{definition}
	A \textit{smooth correspondence} (or just correspondence) between smooth manifolds $X$ and $Y$ (with $X$ compact) is a quadruple $\Phi = (M, b, f, \xi)$, often depicted as a diagram
		\[\Phi = X\corl{b}(M,\xi)\corr{f}Y, \]
	where
		\begin{itemize}
			\item $M$ is a smooth manifold, 
			\item  $\xi \in \K^*(M)$,
			\item $b:M\to X$ is a smooth map, and
			\item $f:M\to Y$ is smooth and K-oriented: that is, \(T_f:= TM \oplus f^*(TY)\) is a \(\K\)-oriented (real) vector bundle. 
		\end{itemize}
	We sometimes refer to $f$ as the \textit{forward map} and $b$ the \textit{backward map} of $\Phi$. The \emph{sum} of two correspondences is defined by their disjoint union. The \emph{degree} of a correspondence is defined to be $\dim Y + \deg \xi-\dim M$ if this locally constant function is constant; any correspondence is the sum of correspondences with well defined degrees. 
\end{definition}

\begin{remark}
\label{remark:OrientationsAndK-orientations}
	Throughout this article we will blue the distinction between K-orientations and orientations on trivial vector bundles. Any K-oriented vector bundle $V$ can be canonically oriented, by the converse is not true unless $V$ is trivial. A tool for building both orientations and K-orientations on vector bundles is the 2-out-of-3 lemma, which is stated as follows. 
	
	\begin{lemma}[The 2-out-of-3 Lemma]
		Suppose that 
		\[0\corr{ } V \corr{ } U \corr{ }  W \corr{ } 0\]
		is a short exact sequence of vector bundles. If any two of them are $\K$-oriented (resp. oriented), then the third receives a canonical $\K$-orientation (resp. oriented). 
		
		Furthermore, if $V$ and $U$ are $\K$-oriented (resp. oriented), then the induced $\K$-orientation (resp. orientation) on $W$ has the property that $V\oplus W$ is isomorphic to $U$ as $\K$-oriented (resp. oriented) vector bundles and the induced $\K$-orientation (resp. orientation) on $W$ is the unique one with this property. 
	\end{lemma}
	
	If $V$ and $W$ (and therefore $U$) are trivial bundles, then it follows from the uniqueness property of the 2-out-of-3 lemma that ``the following diagram commutes'':
	\begin{equation*}
	\xymatrix{
		{\rm Orientation} \ar[r] \ar[d]_{2\text{-out-of-}3} & \text{K-Orientation} \ar[d]^{2\text{-out-of-}3} \\
		{\rm Orientation} \ar[r]  & \text{K-Orientation}.
	}
	\end{equation*}
	Because of this, we can can talk exclusively about orientations when using the composition product (see Section~\ref{section:IntersectionProduct}) so long as all the bundles we are dealing with are trivial. 
\end{remark}

\begin{example}
	The \emph{zero correspondence} is the correspondence obtained by taking $M=\emptyset$.
\end{example}

\begin{example}
	If $\Psi = (M,b,f,\xi)$ is a correspondence between $X$ and $Y$ then its \emph{negative} is the correspondence $-\Psi:= (M,b,-f,\xi)$, where $-f$ means the same map as $f$ but with the opposite K-orientation.
\end{example}

If  $X$ and $Y$ are smooth manifolds, the group \(\KK_*(X,Y)\) is defined as the collection of (smooth) correspondences between $X$ and $Y$, modulo equivalence. Equivalence of correspondences is generated by three steps:
	\begin{quote}
		$\bullet$ Isomorphism, \\
		$\bullet$ Thom modification, and \\
		$\bullet$ Bordism.
	\end{quote}

The notion of isomorphic correspondences is clear, so we now define the other two notions of equivalence. 

\begin{definition}
	If $V$ is a (real) K-oriented vector bundle over $M$ and $\tau^V_*:\K^*(M)\to \K^*(V)$ is the Thom Isomorphism, then we define the \textit{Thom modification of $(M,b,f,\xi)$ along $V$} to be the correspondence 
	\[X\corl{b\circ\pi_v}(V,\tau^V_*(\xi))\corr{f\circ\pi_v} Y. \]
	Here we are K-orienting $f\circ\pi_v$ as the composition of K-oriented maps. If $\Phi$ is a Thom modification of $\Psi$ along some vector bundle, then we write $\Psi \sim_{Tm} \Phi$.
\end{definition}

Next we define the appropriate notion of bordism of correspondences, beginning with the definition of a $\partial$-correspondence.

\begin{definition}
	A \textit{correspondence with boundary}, or \textit{$\partial$-correspondence}, is a correspondence $(M,b,f,\xi)$, where $M$ is a manifold with boundary. 
\end{definition}

Suppose that $\partial M= \partial_0M \sqcup \partial_1M$. Give the inward facing normal bundle at $\partial_0M$ the positive K-orientation and the inward facing normal bundle at $\partial_1M$ the negative K-orientation. Such a $\partial$-correspondence $\Phi$ induces a correspondence on its boundary as 
 \[\partial_i\Phi := X \corl{b|_{\partial_i M}} (\partial_i M,\xi |_{{\partial_iM}})\corr{f|_{\partial_iM}} Y; \]
here we give $f|_{\partial_iM}$ the K-orientation coming from the 2-out-of-3 Lemma.

\begin{definition}
\label{Bordism}
	Two smooth correspondences $\Phi_0$ and $\Phi_1$ are called \textit{bordant} if there is a $\partial$-correspondence $\Phi$ such that 
		\[(-1)^i\Phi_i = \partial_i \Phi \quad \text{for $i=0,1$.} \]
	If $\Phi$ and $\Psi$ are bordant correspondences we write $\Phi\sim_b\Psi$. 
\end{definition}

Putting isomorphism, Thom modification, and bordism together yields equivalence of correspondences. 

\begin{definition}
	For smooth manifolds $X$ and $Y$ with $X$ compact, we define $\KK_* (X, Y)$ to be the set of equivalence classes of correspondences from $X$ to $Y$, where equivalence of correspondences to be the equivalence relation generated by isomorphism, Thom modification, and bordism. We denote the class of a correspondence $(M,b,f,\xi)$ in $\KK_*(X,Y)$ by 
		\[[X \corl{b} (M,\xi) \corr{f} Y]. \]
\end{definition}

It is shown in~\cite{EM:Geometric_KK} that there is a canonical map \(\KK_*(X,Y) 
\to \KK_*(C(X), C(Y))\) from geometric to analytic KK, which is an isomorphism. 

In what follows all of our calculations will tacitly take place in geometric $\KK$. In addition, we need the following two lemmas which will aid our later computations. 

\begin{lemma}
\label{excision}
	Let $\Phi = (M,\xi, b, f)$ be a correspondence from $X$ to $Y$. Let $N\subset M$ be an open subset and suppose there is some $\eta \in \K^*(N)$ which maps to $\xi$ under the map \(\K^*(N) \to \K^*(M)\) induced by the open inclusion. Then 
		\[[X\corl{b}(M,\xi)\corr{f}Y] = [X\corl{b|_N}(N,\eta)\corr{f|_N}Y].\] 	
\end{lemma}
\begin{proof}
	See Example 2.13 in~\cite{EM:Geometric_KK}. 
\end{proof}

\begin{lemma}
\label{lemma: addition}
	We have that 
		\[[X\corl{b}(M,\xi_0) \corr{f}Y]+[X\corl{b}(M,\xi_1) \corr{f}Y] = [X\corl{b}(M,\xi_0+\xi_1) \corr{f}Y] \]
	in $\KK_*(X,Y)$. 
\end{lemma}
\begin{proof}
	See Lemma 2.19 in~\cite{EM:Geometric_KK}. 
\end{proof}

In the rest of the paper, we will abbreviate notation for correspondences involving a point as targets or source: classes
	\[[\pnt \leftarrow (M, \xi) \xrightarrow{f} Y] \quad \text{and} \quad  [X \xleftarrow{b} (M, \xi) \to \pnt]\]
will be simply denoted 
	\[[(M, \xi) \xrightarrow{f} Y] \quad \text{and} \quad  [X\xleftarrow{b} (M, \xi)],\]
respectively.

\subsection{The Intersection Product for Transverse Correspondences}
\label{section:IntersectionProduct}

The most important features of Kasparov's bivariant K-theory is the existence of the external product 
	\[\hotimes_B:\KK_*(A,B)\times\KK_*(B,C) \to \KK_*(A,C),\]
	which gives KK the structure of a category. 
	 
The observation of Connes and Skandalis \cite{Connes-Skandalis} is that the external product
is easy to describe in geometric KK. 

\begin{definition}
	Suppose that $\Phi = (M,\xi,b_M,f_M)$ and $\Psi=(N,\eta,b_N,f_N)$ are \textit{transverse} if the map 
		\[df_M-db_N:T_mM\oplus T_nN \to T_{f_M(m)}Y \]
	is surjective for all $(m,n)\in M\times_YN:= \{(x,y)\in M\times N:f_M(x)=b_N(y)\}$.
\end{definition}

Transversality ensures that the fibered product $M\times_YN$ is a smooth manifold. Let $\pr_M:M\times_YN\to M$ (resp. $\pr_N$) is the projection onto $M$ (resp. $N$).

\begin{definition}
	If $\Phi = (M,\xi,b_M,f_M)$ and $\Psi=(N,\eta,b_N,f_N)$ are transverse correspondences from $X$ to $Y$ and from $Y$ to $Z$ respectively, then their \textit{intersection product} is the correspondence
		\[ \Phi \otimes \Psi := [X \corl{b_M\circ\pr_M}(M\times_YN,\pr_M^*(\xi)\hotimes \pr_N^*(\eta)) \corr{f_N\circ \pr_N} Z]\in \KK_*( X,Z) , \]
	where $f_N\circ \pr_N$ is given the K-orientation discussed below. 
\end{definition}

We endow $f_N\circ \pr_N$ with the following K-orientation. Since $f_N$ is K-oriented, it is sufficient to K-orient $\pr_N$ as the composition of K-oriented maps is K-oriented. To do this, first observe that transversality implies that we have an exact sequence
	\[0 \corr{ } T(M\times_Y N) \corr{ } T(M\times N)|_{M\times_YN} \corr{df_M-db_N} (f_M\circ \pr_M)^*TY \corr{} 0. \]
In particular, we have an isomorphism  $T(M\times_Y N)\oplus (f_M\circ \pr_M)^*TY\cong \pr_M^*TM\oplus \pr_N^*TN$. This implies the equality
	\begin{align*}
		T(M\times_Y N)\oplus\pr_N^*(TN) & \oplus \pr_M^*(TM\oplus f_M^*TY) \\
		& \cong \pr_M^*(TM\oplus TM)\oplus \pr_N^*(TN\oplus TN)
	\end{align*}
which, since $f_M$ is K-oriented, gives a K-orientation to $\pr_N$ by the 2 out of 3 lemma. Two different choice of splitting will yield isomorphisms which are connected by a path, whence give the same K-orientation. 

\subsection{Poincar\'{e} Duality}

Poincar\'e duality (or \emph{spin} duality, as we call it here) follows almost trivially from the 
description of the intersection product using transversality. 

 Before we state it, we need to define several more operations in $\KK$. 

\begin{definition}
	Let $\Phi = (M,\xi,b_M,f_M)\in \KK^*(X_1,Y_1)$ and $\Psi = (N,\eta,b_N,f_N)	\in \KK_*(X_2,Y_2)$. The \textit{exterior product} of $\Phi$ and $\Psi$ is defined as 
		\[\Phi\times \Psi := [X_1\times X_2 \corl{b_M\times b_N} (M\times N,\pr_M^*(\xi)\hotimes\pr_N^*(\eta)) \corr{f_M\times f_N} Y_1\times Y_2] \in \KK_*(X_1\times X_2, Y_1\times Y_2), \]
	where $f_M\times f_N$ is given the product K-orientation. 
\end{definition}

Combining the intersection product with the exterior product we obtain natural cup-cap products ``over" any auxiliary space $U$ by: 
	\begin{align}
	\label{align:cupcapproduct}
	\begin{split}
		\hotimes_U:\KK_*(X_1,Y_1\times U) \times \KK_*(U\times X_2,Y_2)  & \to \KK_*(X_1\times X_2, Y_1\times Y_2) \\
		(\Phi,\Psi) & \mapsto (\Phi\times \id_{X_2})\otimes_{X_2\times U \times Y_1} (\id_{Y_1}\times \Psi).
	\end{split}
	\end{align}
	
If $X$ is a compact, K-oriented manifold and $\delta:X\to X\times X$ is the diagonal map, then we define the class 
	\[[\Dudeltaspin] := [(X,\mathbf{1}) \corr{\delta} X\times X] \in \KK_{\dim(X)}(\pnt, X\times X).  \]

\begin{proposition}[Poincar\'{e} Duality]
	Let $X$ be compact and K-oriented. For any $Y$ and $Z$ there is a natural isomorphism
		\[ \PDspin:\KK_*(X\times Y, Z) \to \KK_{*+\dim(X)}(Y,X\times Z) \]
	given on correspondences by
		\[[X\times Y \corl{b_X\times b_Y} (M,\xi) \corr{f} Z] \mapsto [ Y \corl{b_Y} (M,\xi) \corr{b_X\times f} X\times Z], \]
	where $b_1\times f$ is K-oriented using the isomorphism 
		\[TM\oplus(b_1\times f)^*T(X\times Z) \cong b_X^*TX\oplus(TM\oplus f^*TZ).  \]
\end{proposition}

For a more detailed discussion see Section~\ref{section:Baum-Connes} or~\cite{Connes-Skandalis}. In particular, if $M$ and $X$ are parallelizable manifolds K-oriented using orientations and $Z$ is a point, then we have 
	\begin{corollary}
		The map 
			\begin{align*}
				\PDspin:\KK_*(X\times Y,\pnt) & \to \KK_{*+\dim(X)}(Y,X) \\
				[X\times Y \corl{b_X\times b_Y}(M,\xi)] & [Y\corl{b_Y} (M,\xi) \corr{b_X} X],
			\end{align*}
		where $b_X:M\to X$ is K-oriented using the product orientation on $M\times X$, is an isomorphism. 
	\end{corollary}

\subsection{K-Theory of Tori via Correspondences}

We begin this section by making a simple observation. Let $\delta:X\to X\times X$ be the diagonal map. The ring structure 
	\[\wedge:\KK_{d_0}(\pnt,X)\times \KK_{d_1}(\pnt,X) \to \KK_{d_0+d_1}(\pnt,X)\] 
on topological \(\K\)-theory can be described in $\KK$ using the following commutative diagram 
	\begin{align*}
	\xymatrix{
		\KK_{d_0}(\pnt,X) \times \KK_{d_1}(\pnt, X) \ar[r]^-{\times} \ar[dr]_{\wedge} & \KK_{d_0+d_1}(\pnt, X\times X) \ar[d]^{\hotimes_{X\times X}[\delta]_*} \\
		& \K_{d_0+d_1}(\pnt, X).
	}
	\end{align*}

More geometrically, using correspondences, suppose that \(i_0\colon M_0 \to X\) and \(i_1\colon M_1 \to X\) are two closed submanifolds with K-oriented normal bundles of dimensions $d_0$ and $d_1$, respectively; this is equivalent to $i_0$ and $i_1$ being $\K$-oriented. In this case the submanifolds $M_0$ and $M_1$ define classes 
	\[[i_0]_! = [\pnt \corl{ } (M_0,\mathbf{1}) \corr{i_0} X \in \KK_{d_0}(\pnt, X) \quad \text{and} \quad [i_1]_! = [\pnt \corl{ } (M_1,\mathbf{1}) \corr{i_1} X] \in \KK_{d_1}(\pnt, X).  \]
If $M_0$ and $M_1$ are transverse, then their product $M_0\times M_1$ and the diagonal in $X\times X$ are transverse and we can use the intersection product to compute the composition $[i_0\times i_1]_!\hotimes_{X\times X}[\delta]_*$. The composition diagram is 
	\begin{align*}
	\xymatrixrowsep{10pt}\xymatrixcolsep{15pt}\xymatrix{
		& & M_0\cap M_1 \ar[dl]_{\delta} \ar[dr] & & \\
		& M_0\times M_1 \ar[dl] \ar[dr] & & X \ar[dl]^-{\delta} \ar[dr]	& \\
		\pnt & & X\times X & & X,
	}
	\end{align*}
from which we deduce 
	\begin{proposition}
	\label{proposition:ringstructureintermsofcorrespondences}
		If $i_0:M_0\into X$ and $i_1:M_0\into X$ are transverse closed submanifolds of codimension $d_0$ and $d_1$, respectively, with K-oriented normal bundles, then their product $[i_0]_!\wedge [i_1]_! \in \KK_{-d_0-d_1}(\pnt, X)$ is given by the correspondence 
			\[[(M_0\cap M_1,\mathbf{1}) \corr{i} X]\in \KK_{d_0+d_1}(\pnt,X)\cong \K^{-d_0-d_1}(X), \]
		where $i:M_0\cap M_1\to X$ is the inclusion.
	\end{proposition}

In the spirit of Proposition~\ref{proposition:ringstructureintermsofcorrespondences}, we now apply the correspondence framework to the K-theory and K-homology of tori. We start with some basic remarks. The \(\K\)-theory of the torus \(\T^d\)  can 
be identified with an exterior algebra 
	\[ \K^*(\T^d) \cong \Lambda_\Z^* (\Z \{x_1, \ldots , x_n\}) \]
of a free abelian group \(\Z\{x_1, \ldots , x_n\}\) on \(d\) generators; the quickest way to see this is via the K\"unneth Theorem, which implies that 
	\[ \K^*(\T^d) \cong \K^*(\T) \hotimes_\Z \K^*(\T)\hotimes_\Z \cdots \hotimes_\Z \K^*(\T),\]
where \(\hotimes\) is the graded tensor product of groups. As \(\K^*(\T)\) has two generators, \([1]\in \K^0(\T)\) and \([u]\in \K^{-1}(\T)\), the external products 
	\[x_k := [1] \hotimes_\Z  \cdots \hotimes_\Z [1] \, \hotimes_\Z\,  [u] \, \hotimes_\Z [1] \hotimes_\Z \cdots \hotimes [1] \in \K^{-1} (\T^d),\] 
where the $[u]$ term is in the $k$-th entry of the product, give generators \(x_k\) of $\K^*(\T^d)$ in the sense that \emph{products} (corresponding to the ring structure on \(\K^*(\T)\)) give abelian group generators 
	\[ x_{i_1}\wedge \cdots \wedge x_{i_r} \in \K^{-r}(\T^d), \;\;\; i_1<\cdots <i_r.\]
Since \(x_i \wedge x_j = -x_j \wedge x_i\), this describes \(\K^*(\T^d)\) as an exterior algebra. 

In terms of correspondences, the unitary \(u(z) = z\) on \(\T\) represents the `Bott class' of the circle, which is the class of the \(1\)-point correspondence \([(\pnt,\mathbf{1}) \rightarrow \T],\) by including, say, the point \(1\in \T\). If $\pi_k:\T^d\to \T$ is the projection onto the $k$-th component, then it follows that the \(x_k\) are represented by the compositions 
	\begin{equation*}
	\xymatrixrowsep{10pt}\xymatrixcolsep{15pt}\xymatrix{
		& & \pi_k^{-1}(1)  \ar[dl] \ar[dr] & & \\
		& \pnt \ar[dr] \ar[dl] & & 	\T^d \ar[dl]_{\pi_k} \ar[dr]^{} & \\
		\pnt & & \T & & \T^d.
	}
	\end{equation*}
The subtori \( T_k := \pi_k^{-1}(1)\), \(k=1, \ldots d\) are hyperplanes in \(\T^d\) and the conclusion is that these hyperplanes represent ring generators of $\K^{-1}(\T^d)$ corresponding to the classes
	\[x_k = \left[ (T_k,\mathbf{1}) \corr{i_k} \T^d\right] \in \KK_{+1}(\ast \;, \T^d)\cong \K^{-1}(\T^d),\]
where \(i_k \colon T_k\to \T^d\) the inclusion. However, the ring product \(\wedge\) on \(\K^*(\T^d)\) corresponds to intersection of cocycles as explained in Proposition~\ref{proposition:ringstructureintermsofcorrespondences}. If $r\neq s$, then the subtori $T_r$ and $T_s$ are clearly transverse, to we may for instance interpret the ring product \(x_r \wedge x_s\) geometrically by
	\[ x_r \wedge x_s = \left[(T_r\cap T_s,\mathbf{1}) \corr{i_{r,s}} \T^d\right]\]
with \(i_{r,s}\) the inclusion. The intersection \(T_r\cap T_s\) is now a co-dimension \(2\) subtorus, and so we see that $x_r\wedge x_s$ defines a class in $\KK_{+2}(\pnt, \T^d)$. 

One may iterate this construction in the obvious way: if \(T_{k_1} , \ldots , T_{k_r},\) are any \(r\) distinct coordinate hyperplanes in \(\T^d\), and if we partition \(\{k_1 < \ldots < k_r\}\) into two sets \(I\) and \(J\), then the inclusion of \(T_I:=\cap_{k_i \in I} T_{k_i}\) in \(\T^d\) is tranverse to the inclusion of \(T_J:=\cap_{k_i \in J} T_{k_i}\) in \(\T^d\), provided that \(r\le d\), which implies by transversality that  
	\[[(T_I\cap T_J,\mathbf{1}) \corr{i_{I\cap J}} \T^d] = x_I\wedge x_J = (-1)^{\sigma(I,J)}x_{k_1}\wedge \cdots \wedge x_{k_r}, \]
where $i_{I\cap J}: T_I\cap T_J\to \T^d$ is the inclusion and $(-1)^{\sigma(I,J)}$ is the sign of the permutation taking $I\cup J$ to $\{k_1 < \ldots < k_r\}$. 

More generally, we want to consider \emph{oriented subtori} of \(\T^d\). 
	\begin{definition}
		An \emph{oriented subtorus} $(T,i)$ of \(\T^d\) is a Lie embedding 
			\[ i \colon T \to \T^d\]
		of a closed subgroup of \(\T^d\), together with an orientation, whence \(\K\)-orientation, on \(T\).
	\end{definition}
An oriented subtorus $(T,i)$ of $\T^d$ defines classes 
	\[ [(T,i)]_* := \left[ \T^d\corl{i} (T, \mathbf{1}) \right] \in \KK_{-k}(\T^d, \pnt) \quad \text{and} \quad  [(T,i)]_!:=\left[(T,\mathbf{1}) \corr{i} \T^d\right]\in \KK_{d-k}(\pnt , \T^d),\]
where \(k = \dim (T)\). If \(i'\colon T' \to \T^d\) is another (oriented) subtorus such that \(i,i'\) are transverse, then the composition diagram 

	\begin{equation*}
	\xymatrixrowsep{10pt}\xymatrixcolsep{15pt}\xymatrix{
		& &  T\cap T'  \ar[dl] \ar[dr] & & \\
		& T \ar[dr]_{i}  \ar[dl] & & 	T' \ar[dl]^{i'} \ar[dr]^{} & \\
		\ast& &\T^d & & \ast.
	}
	\end{equation*}
describes the \(\K\)-theory \(\K\)-homology pairing between \([(T,i)]_!\) and \([(T',i')]_*\) in terms of a correspondence from a point to a point: 
	\[ \langle [(T,i)]_!, [(T',i')]_*\rangle := [(T,i)]_!\hotimes_{\T^d} [(T',i')]_* = \left[\pnt \corl{ } (T\cap T',\mathbf{1})\rightarrow \pnt\right] \in \KK(\pnt, \pnt) \cong \Z.\]
Now any torus of positive dimension has exactly two \(\K\)-orientations corresponding to its two orientations, and it is a boundary with either of them. Since every connected component of \(T\cap T'\) is a torus, $T\cap T'$ is a boundary and therefore the correspondence \(\pnt \leftarrow T\cap T' \rightarrow \pnt\) is a $\partial$-correspondence whence  
	\[ \langle [(T,i)]_!, [(T',i')]_*\rangle = [\pnt \corl{ } (T\cap T',\mathbf{1}) \corr{ } \pnt] = 0,\]
unless \(T\) and \(T'\) have exactly complementary dimension, which condition ensures that \(T\cap T'\) is zero-dimensional. To summarize, 

	\begin{proposition}
		If \((T,i)\) and \((T',i')\) are oriented closed tori of \(\T^d\), of dimensions \(k,k'\), then the \(\K\)-theory-\(K\)-homology pairing \( \langle [T]_!, [T']\rangle\) between \([T]_!\) and \([T']\) is zero unless \(k+k' = d\), and in this case, 
			\[ \langle [T]_!, [T']\rangle = | T\cap T'|,\]
 		where \(|T\cap T'|\) is the number of points in the intersection. 
	\end{proposition}

The Proposition together with an obvious guess supplies a natural dual basis is to the \(x_i\)'s, in the sense of the \(\K\)-theory-\(\K\)-homology pairing. Let 
	\[ y_k := \left[ \T^d\corl{j_k} (\T,\mathbf{1})\right] \in \KK_{-1}(\T^d, \; \pnt),\]
where $j_k:\T\to \T^d$ is the inclusion into the $k$-th coordinate. Computing \(\langle x_k, y_k\rangle\) by transversality gives immediately that \( \langle x_k, y_l\rangle = \delta_{ij} \) so that \(y_1, \ldots , y_d\) is the dual basis. We may thus identify the \(\K\)-homology with the abelian group \( \K_*(\T^d) \cong \Lambda^*_\Z (\Z\{ y_1, \ldots y_n\}).\)

\section{The Fourier-Mukai Transform}

Let $\T^d$ denote the $d$-dimensional torus and $\widehat{\Z}^d:={\rm Hom}(\Z^d,\T)$ denote the Pontryagin of $\Z^d$. The \emph{Fourier-Mukai transform} is a canonical class $[\mathcal{F}_d] \in \KK_{-d}(\T^d,\widehat{\Z}^d)$, and this section geometrically describes the map $[\mathcal{F}_d]\hotimes_{\widehat{\Z}^d}:\K^*(\widehat{\Z}^d)\to \K^{*-d}(\T^d)$. We begin by defining the K-theory data in $[\mathcal{F}_d]$. 

	\begin{definition}
		The \emph{Poincar\'{e} bundle} is the complex line bundle over $\T^d\times \widehat{\Z}^d$ given by 
			\[\mathcal{P}_d = (\R^d\times \widehat{\Z}^d)\times_{\Z^d}\C := [(\R^d\times \widehat{\Z}^d)\times\C]/\sim, \]
		where $\sim$ is the relation 
			\[(x,\chi, \lambda) \sim (x+n, \chi, \chi(n)\lambda) \quad \text{for $n \in \Z^d$}. \]
		The bundle projection is the map induced by the coordinate projection $(\R^d\times \widehat{\Z}^d)\times \C \to \R^d\times \widehat{\Z}^d$. 
	\end{definition}

A trivializing neighbourhood around any $(z,\chi)\in \T^d\times \widehat{\Z}^d$ an be given as follows. Let $q:\R^d \to \T^d=\R^d/\Z^d$ be the quotient map, and let $U\subset \T^d$ be an open neighbourhood of a point $(z,\chi)\in \T^d\times \widehat{\Z}^d$ such that $q^{-1}(U)$ is a countable, disjoint union of open sets $U_k$, each of which containing a unique point $k\in \Z^d$ in the integer lattice. The map 
	\begin{align*}
		\tilde{\sigma}:\bigsqcup_{k\in \Z^d}(U_k\times \widehat{Z}^d) & \to \mathcal{P}_d|_{U\times \widehat{Z}^d} \\
		(x,\chi) & \mapsto [(x,\chi, \chi(k))] \quad \text{for $x\in U_k$}
	\end{align*}
satisfies $\sigma(x+n,\chi) = \sigma(x,\chi)$ and therefore descends to a non-vanishing section $\sigma:U\times \widehat{\Z}^d \to \mathcal{P}_d|_{U\times \widehat{\Z}^d}$, showing that $\mathcal{P}_d$ is trivial on $U\times \widehat{\Z}^d$.

	\begin{definition}
		The $d$-dimensional \emph{Fourier-Mukai transform} is the class $[\mathcal{F}_d]\in \KK_{-d}(\T^d,\widehat{\Z}^d)$ defined by the correspondence
			\[\T^d \corl{\pr_1}(\T^d\times \widehat{\Z}^d,\mathcal{P}_d) \corr{\pr_2} \widehat{Z}^d, \]
		where the coordinate projection $\pr_2:\T^d\times \widehat{\Z}^d \to \widehat{\Z}^d$ is K-oriented by using the canonical product orientations on $\T^d\times \widehat{\Z}^d$ and $\widehat{\Z}^d$. 
	\end{definition}  

By a \emph{torus} we mean a quotient \(V/\Gamma\) where \(V\) is a real vector space and \(\Gamma \subset V \) is a lattice in \(V\). We say that $(T = V/ \Gamma,i)$ is an \emph{oriented subtorus} of $\T^d$ if $T$ is oriented and $i:T\to \T^d$ is an embedding of Lie groups. If $(T,i)$ is an oriented subtorus of $\T^d$ of dimension $k$, then $(T,i)$ defines a class $[(T,i)]_!\in \KK_{d-k}(\pnt,\T^d)$ via the correspondence
	\[[(T,i)]_! := [\pnt \corl{ } (T,\mathbf{1}) \corr{i} \T^d], \]
where $i:T\to \T^d$ is oriented using the orientation on $T$ and the product orientation on $\T^d$. 

We will now show that $(T,i)$ also defines a \emph{dual class} $[\widehat{(T,i)}]_! \in \KK_{d-k}(\pnt,\widehat{\Z}^d)$. We first note that differentiating $i$ at the identity $e\in T$ gives a linear injection $d_ei:V\to \R^d$. Since the quotient map $q:V\to V/\Gamma$ is the Lie group exponential map, it follows furthermore that the following diagram commutes

	\begin{equation*}
		\xymatrix{
			V \ar[r]^{d_ei} \ar[d] & \R^d \ar[d] \\
			T \ar[r]_i & \T^d, 
	}
	\end{equation*}
	where the vertical maps are the corresponding quotient maps. Because of this, we see that $d_ei$ maps $\Gamma$ injectively into a subgroup of $\Z^d$. The Pontryagin dual of $d_ei$ is therefore a surjective homomorphism $\widehat{d_ei}:\widehat{\R}^d\onto \widehat{V}$ mapping $\Z^d_\perp$ into $\Gamma_\perp$. It follows that $\widehat{d_ei}$ descends to a group surjection 
	\[\hat{i}:\That^d \onto \widehat{V}/\Gamma_\perp \cong \widehat{\Gamma}.  \]

	\begin{definition}
		The \emph{dual subtorus} of the embedded subtorus $T\overset{i}{\into} \T^d$ is the kernel of $\widehat{i}:\That^d\to \widehat{\Gamma}$:
			\[\widehat{T} := \ker\{\hat{i}:\widehat{\Z}^d\to \widehat{\Gamma}\} \subseteq \widehat{\Z}^d. \]
	\end{definition}

By definition, $\widehat{T}$ fits into the following exact sequence 

	\begin{equation}
	\label{DualTorusExactSequence1}
		0 \corr{ } \widehat{T} \corr{ } \widehat{\Z}^d \corr{\hat{i}} \widehat{\Gamma} \corr{ } 0,
	\end{equation}
	which endows $\widehat{T}$ with an orientation, by the 2-out-of-3 lemma. Using this, we can define the dual class of $(T,i)$. 

	\begin{definition}
		If $(T,i)$ is an oriented sub-torus of $\T^d$, then its dual class $[\widehat{(T,i)}]_! \in \KK_{k}(\pnt, \That^d)$ is given by the correspondence 
			\[[\widehat{(T,i)}]_!:=[\pnt \corl{}(\widehat{T},\mathbf{1}) \corr{ } \That^d ], \]
		where the embedding $\widehat{T} \into \That^d$ is oriented according to~\eqref{DualTorusExactSequence1}. 
	\end{definition}

With this notation behind us, we can state the main theorem of this section. 

	\begin{theorem}
	\label{thm: MainK-theoryComputation}
		Suppose that $(T,i)$ is a $k$-dimensional oriented, embedded subtorus of $\T^d$ with dual subtorus $\widehat{T}\into \That^d$. Then 
			\begin{align*}
				[(T,i)]_!\hotimes_{\T^d}[\mathcal{F}_d] &= (-1)^{k(d-k)+\frac{k(k-1)}{2}}\cdot [\widehat{(T,i)}]_!\in \KK_{k}(\pnt,\That^d).
			\end{align*}
	\end{theorem}

In order to prove Theorem~\ref{thm: MainK-theoryComputation}, we will compute the intersection product and then perform a Thom modification. The first observation to make, which is fundamental to our argument, is the following.

	\begin{lemma}
		For $x\in (0,1)$ let $e^x=-\exp(2\pi i x)\in \T$, let $\widehat{e^x} \in \That$ be the character
			\[\widehat{e^x}:n\mapsto e^{nx} \quad \text{for $n\in \Z$, } \]
		and let $e_!:\K^0((0,1)^2)\to \K^0(\T\times \That)$ be the shriek map induced by $e(x,y) = (e^x, \widehat{e^y})$. If $\beta \in \K^0((0,1)^2)\cong \K^0(\R^2)$ denotes the Bott class, then we have 
			\[e_!(\beta) = [\mathcal{P}_1]-[\mathbf{1}] \in \K^0(\T\times \That).  \]
	\end{lemma}
	\begin{proof}
		Let ${\rm ev}_1:\That \to \T$ denote the orientation preserving diffeomorphism given by evaluation at 1; this fits into the following commutative diagram
			\begin{equation*}
			\xymatrixcolsep{45pt}\xymatrix{
				\K^0(\T\times (0,1)) \ar[r]^{({\rm ev}_1,\id)^*} & \K^0(\Zhat\times (0,1)) \\
				\K^{0}((0,1)^2) \ar[u]^{(e^x,\id)_!} \ar[ur]_{(\widehat{e^x},\id)_!}.
			}
			\end{equation*}
		Since $(e^x,\id)_!(\beta) = [z]\in\K^{-1}(\T)=\K^0(\T\times (0,1))$ and ${\rm ev}_1^*([z]) = [\chi]\in \K^{-1}(\Zhat) = \K^0(\Zhat \times (0,1))$, where $[\chi]$ denotes the class of the unitary $\chi(1)$, it follows that $(\widehat{e^x},\id)_!(\beta) = [\chi]$. Since the following diagram commutes
			\begin{equation*}
			\xymatrixcolsep{45pt}\xymatrix{
				\K^0(\Zhat\times (0,1)) \ar[r]^{(\id,e^x)_!} & \K^0(\Zhat\times \T) \ar[r]^{{\rm flip}} & \K^0(\T\times \Zhat)	\\
				\K^0((0,1)^2) \ar[u]^{(\widehat{e^y},\id)_!} \ar[ur]|-{(\widehat{e^y},e^x)_!} \ar@/_1pc/[urr]_{e_!}
			}
			\end{equation*}
		it remains to show that $(\id,e^x)_!([\chi]) = [\mathcal{P}_1]-[\mathbf{1}]\in \K^0(\Zhat\times \T)$.
	
		We identify the 1-point compactification of $\Zhat\times (0,1)$ with the space $(\Zhat\times \T)/(\Zhat \times \{-1\})$. It follows that the map $(\id,e^x)_!$ is the restriction of the quotient map 
			\[q:\Zhat\times \T \to  (\Zhat\times \T)/(\Zhat\times\{-1\}) = (\Zhat\times [0,1/2])/(\Zhat\times \{1/2\})\cup(\Zhat\times [1/2,1])/(\Zhat\times \{1/2\})\]
		to the kernel of the augmentation map $\{\infty\} \to [\Zhat\times (0,1)]^+$. The class of $[\chi] \in \K^0((0,1)\times \Zhat)$ is represented in $\K^0([\Zhat\times (0,1)]^+)$ by the class $[E_\chi] - [\mathbf{1}]$, where $E_\chi$ is, by definition, the trivial line bundle modulo the relation 
			\begin{equation}
			\label{rel:poincare}
				(\chi, 0, \lambda) \sim (\chi, 1, \chi(1)\lambda).
			\end{equation}
		It is easy to see that $q^*(E_\chi) = [\Zhat\times [0,1]]\times \C/\sim$, where $\sim$ is the relation in~\eqref{rel:poincare}. This is clearly isomorphic to the Poincar\'{e} bundle on $\Zhat\times \T$, which concludes the proof. 
	\end{proof}

Using this, we have the following version of Bott periodicity. 
	\begin{lemma}
	\label{Lemma:BottPeriod}
		We have that 
			\[[(\T\times \Zhat, \mathcal{P}_1)\corr{\pr_2} \Zhat] = [(\pnt,\mathbf{1})\corr{ } \Zhat] \]
		in $\KK_*(\pnt, \Zhat)$. 
	\end{lemma}
	\begin{proof}
		The correspondence $(\T\times \Zhat,\mathbf{1}) \corr{\pr_2} \Zhat$ is a $\partial$-correspondence, so we have that 
			\[[(\T\times \Zhat, \mathcal{P}_1)\corr{\pr_2} \Zhat] = [(\T\times \Zhat, \mathcal{P}_1-\mathbf{1})\corr{\pr_2} \Zhat] \]
		in $\KK_{-1}(\pnt,\Zhat)$. Using the previous lemma and Example~\ref{excision}, we have that 
			\[[(\T\times \Zhat, \mathcal{P}_1-\mathbf{1})\corr{\pr_2} \Zhat] = [((0,1)^2, \beta)\corr{\widehat{e^{\pr_2}}} \Zhat]\sim_b[((0,1)^2,\beta) \corr{1} \Zhat],\] 
		where $1:(0,1)^2\to \Zhat$ is the constant map at the trivial character. Since $(0,1)^2$ is a tubular neighbourhood of $(\frac{1}{2},\frac{1}{2})\in \R^2$, we have via a Thom modification that 
			\[[((0,1)^2,\beta) \corr{1} \Zhat]\sim_{Tm} [(\pnt, \mathbf{1}) \corr{ } \Zhat]. \]
		Putting this together yields the result. 
	\end{proof}
More generally, define
	\begin{align*}
		\theta:\T^d\times \Zhat^d & \to (\T\times \Zhat)^d \\
		(z_1,\dots, z_d,\chi_1,\dots, \chi_d) & \mapsto (z_1,\chi_1,\dots, z_d,\chi_d). 
	\end{align*}
Then $\det(d\theta) = (-1)^{\frac{d(d-1)}{2}}$ and $\theta^*(\mathcal{P}_d) = \hotimes_{i=1}^d\mathcal{P}_1$ so be repeated application of Lemma~\ref{Lemma:BottPeriod} we have the following corollary. 

	\begin{lemma}
		For any $d$, 
			\[[(\T^d\times \Zhat^d,\mathcal{P}_d)\corr{\pr_2} \Zhat] = (-1)^{\frac{d(d-1)}{2}}[(\pnt,\mathbf{1})\corr{ } \Zhat^d]. \qedhere \]
	\end{lemma}

	\begin{proof}[Proof of Theorem~\ref{thm: MainK-theoryComputation}]
		Since the coordinate projection $\pr_{1}:\T^d\times \Zhat^d\to \T^d$ is a submersion, we can compute the product $[(T,i)]_!\hotimes_{\T^d}[\mathcal{F}_d]$ using the intersection product. The composition diagram is
		\begin{align*}
		\xymatrixrowsep{10pt}\xymatrixcolsep{15pt}\xymatrix{
			& & T\times \Zhat^d \ar[dl]_{\pr_T} \ar[dr]^{(i,\id)} & & \\
			& T \ar[dl] \ar[dr]^{i} & & \T^d\times \Zhat^d \ar[dl]_{\pr_1} \ar[dr]^{\pr_2}	& \\
			\pnt & & \T^d & & \Zhat^d.
		}
		\end{align*}
		According to the intersection product, $T\times \Zhat^d$ is oriented via the following exact sequence:
			\[0\corr{ } T_e(T\times \Zhat^d) \corr{ } T_e(T\times \T^d\times \Zhat^d) \corr{d_ei - d_e(\pr_1)} T_e(\T^d) \corr{ } 0 \]
		whence 
			\[[(T,i)]_!\hotimes_{\T^d}[\mathcal{F}_d] = [(T\times \Zhat^d, (i,\id)^*(\mathcal{P}_d))\corr{\pr_1} \Zhat^d]; \]
		we leave it to the reader to check that the orientation induced by the 2-out-of-3 lemma is the canonical one on $\pr_2:T\times \Zhat^d\to \Zhat^d$. It is an easy exercise to check that $(i,\id)^*(\mathcal{P}_d)$ is the bundle $(V\times \Zhat^d)\times_{\Gamma}\C$ over $T\times \Zhat^d$ defined as the trivial bundle $(V\times \Zhat^d)\times \C$ modulo the relation 
			\[(v,\chi, \lambda) \sim (v+\gamma,\chi, \chi(d_ei(\gamma))\lambda) = (v+\gamma,\chi, \widehat{i}(\chi)(\gamma)\lambda )\quad \text{for $v\in V$ and $\gamma\in \Gamma$.} \]
		Now, according to Pontryagin duality and the exact sequence~\eqref{DualTorusExactSequence1} there is an orientation preserving isomorphism $\varphi: \widehat{T} \times \widehat{\Gamma}\cong \Zhat^d$. This implies that the correspondences 
			\[[(T\times \Zhat^d,(V\times \Zhat^d)\times_{\Gamma}\C)\corr{\pr_2}\Zhat^d] \quad \text{and} \quad [(T\times \widehat{T}\times \widehat{\Gamma},\varphi^*((V\times \Zhat^d)\times_{\Gamma}\C))\corr{\varphi\circ\pr_2}\Zhat^d] \] 
		are equal in $\KK_{-k}(\pnt, \Zhat^d)$. According to $\varphi$, and character $\chi\in \Zhat^d$ can be written as $\chi = (\eta, \xi)$ where $\eta\in \widehat{\Gamma}$ and $\xi \in \widehat{T}$. Because of this, and the fact that $\widehat{T} = \ker(\hat{i})$, it follows that 
			\[\varphi^*((V\times \Zhat^d)\times_{\Gamma}\C) \cong \mathcal{P}_{T\times \widehat{\Gamma}}\hotimes \mathbf{1}. \]
		We therefore see that
			\begin{align*}
				[(T\times \widehat{T}\times \widehat{\Gamma},\varphi^*((V\times \Zhat^d)\times_{\Gamma}\C))\corr{\varphi\circ\pr_2}\Zhat^d] & = (-1)^{k(d-k)}[(T\times \widehat{\Gamma}\times \widehat{T},\mathcal{P}_{T\times \widehat{\Gamma}}\hotimes \mathbf{1})\corr{\varphi\circ\pr_2}\Zhat^d] \\
				& = (-1)^{k(d-k)}[(T\times \widehat{\Gamma},\mathcal{P}_{T\times \widehat{\Gamma}})\corr{}\widehat{\Gamma}]\times [\widehat{(T,i)}]_! \\
				& =(-1)^{k(d-k)+\frac{k(k-1)}{2}} [(\pnt,\mathbf{1})\corr{}\widehat{\Gamma}]\times[\widehat{(T,i)}]_!\\
				& = (-1)^{k(d-k)+\frac{k(k-1)}{2}}[\widehat{(T,i)}]_!,
			\end{align*}
		as claimed. 
	\end{proof}

\begin{example}
\label{example:invert torus}
	Consider the embedding
		\begin{align*}
			\overline{\Delta}:\T^d & \to \T^{2d} \\
			z & \mapsto (\overline{z},z).
		\end{align*}
	We observe that $(d_1\overline{\Delta})(x) = (-x,x)$ and thus 
		\begin{align*}
			\widehat{\overline{\Delta}} : \Zhat^{2d} & \to \Zhat^d \\
			(\chi, \eta) & \mapsto \overline{\chi}\otimes \eta, 
		\end{align*}		
	where $\overline{\chi}\otimes \eta(n) = \overline{\chi}(n)\cdot\eta(n)$ is the tensor product of the characters. It follows that the dual torus of $(\T^d,\overline{\Delta})$ is $(\Zhat^d, \Delta)$, where $\Zhat^d$ is oriented in the standard way and $\Delta:\Zhat^{d}\to \Zhat^{2d}$ is the diagonal map. Using Theorem~\ref{thm: MainK-theoryComputation}, we have that 
		\[[(\T^d, \overline{\Delta})]_!\hotimes_{\T^{2d}}[\mathcal{F}_{2d}] = (-1)^{d+\frac{d(d-1)}{2}}[(\Zhat^d, \Delta)]_!.  \]  
	This example with be important when we prove that $[\mathcal{F}_d]$ is invertible. $\qed$
\end{example}

There is a dual picture to this, which proceeds as follows. Let $\hat{i}:\widehat{\Gamma} \into \Zhat^d$ be an embedding of Lie groups. By Pontryagin duality, the dual of $\hat{i}$ is surjection $i:\Z^d\onto \Gamma$, which extends by linearity to $i:\R^d\to V$. This then descends to a map 
\[i:\T^d=\R^d/\Z^d\to T=V/\Gamma. \]

\begin{definition}
	The \textit{dual subtorus} $Z\into \T^d$ of $\hat{i}:\widehat{\Gamma} \into \Zhat^d$ is the kernel of the induced map $i:\T^d\to T$.
\end{definition}

In a completely analogous argument as given in the proof of Theorem~\ref{thm: MainK-theoryComputation}, one has 

\begin{theorem}
	\label{thm: MainK-HomologyComputation}
	If $\hat{i}:\widehat{\Gamma}\into \Zhat^d$ is an oriented sub-torus, then 
	\[[\mathcal{F}_d]\hotimes_{\Zhat^d}[(\widehat{\Gamma},\hat{i})]_* = (-1)^{kd+\frac{k(k-1)}{2}}\cdot[Z]_*\in \KK_{k-d}(\T^d,\pnt) \] 
\end{theorem}

\begin{example}
\label{example:invert dual}
	This example is dual to Example~\ref{example:invert torus}. Consider the embedding 
		\begin{align*}
			\hat{\overline{\delta}} : \Zhat^d & \to \Zhat^{2d} \\
			\chi & \mapsto (\overline{\chi}, \chi).
		\end{align*}
	As in Example~\ref{example:invert torus}, the dual torus to $(\Zhat, \hat{\overline{\delta}})$ is $(\T^d, \Delta)$, where $\T^d$ is oriented in the standard way and $\Delta:\T^d\to \T^{2d}$ is the diagonal map. Thus, by Theorem~\ref{thm: MainK-HomologyComputation} we have that 
		\[[\mathcal{F}_{2d}]\hotimes_{\Zhat^{2d}}[(\Zhat^{d}, \hat{\overline{\delta}})]_* = (-1)^{\frac{d(d-1)}{2}}[(\T^d, \Delta)]_*. \qed \]
\end{example}

We conclude this section by showing that $[\mathcal{F}_d]$ is an invertible element of $\KK_{-d}(\T^d,\Zhat^d)$ by defining its inverse explicitly. 

\begin{definition}
	Define the bundle $\overline{\mathcal{P}}_d$ over $\Zhat^d\times \T^d$ given by the trivial line bundle $(\Zhat^d\times \R^d)\times \C$ modulo the relation
		\[(\chi, x, \lambda) \sim (\chi, x+n, \overline{\chi}(n)\lambda), \quad \text{for $n\in \Z$.} \]
	The \emph{dual Fourier-Mukai tansform} is the class $[\overline{\mathcal{F}_d}]\in \KK_{-d}(\Zhat^d, \T^d)$ defined by the correspondence 
		\[\Zhat^d \corl{\pr_1} (\Zhat^d\times \T^d, \overline{\mathcal{P}}_d) \corr{\pr_2} \T^d, \]
	where $\Zhat^d\times \T^d$ is given its \emph{complex structure}.
\end{definition}

\begin{theorem}
	The Fourier-Mukai correspondence $[\mathcal{F}_d]$ is invertible in $\KK_{-d}(\T^d,\Zhat^d)$. Moreover, we have that 
		\[[\overline{\mathcal{F}}_d]\hotimes_{\T^d}[\F_d]
		 = \mathbf{1}\in \KK_0(\Zhat^d,\Zhat^d)
		  \quad \text{and} \quad [\FM_d]\hotimes_{\Zhat^d}[\overline{\mathcal{F}}_d] =  \mathbf{1} \in \KK_0(\T^d,\T^d). \]
\end{theorem}
\begin{proof}
	 Let $[\overline{\mathcal{F}}^\times_d]$ denote the dual Fourier-Mukai transformation where the middle manifold is K-oriented using the product orientation. Then we have that $[\overline{\mathcal{F}}^\times_d] = (-1)^{\frac{d(d-1)}{2}}[\overline{\mathcal{F}}_d]$ and so $[\overline{\mathcal{F}}_d]\hotimes_{\T^d}[\FM_d] = (-1)^{\frac{d(d-1)}{2}}[\overline{\mathcal{F}}^\times_d]\hotimes_{\T^d}[\FM_d]$.
	 
	 The intersection diagram for the composition $[\overline{\mathcal{F}}^\times_d]\hotimes_{\T^d}[\FM_d]$ is 
	 	\begin{equation*}
	 	\xymatrixrowsep{10pt}\xymatrixcolsep{15pt}\xymatrix{
	 		& & \T^d\times \Zhat^d\times \Zhat^d \ar[dl]_{\pr_{2,1}} \ar[dr]^{\pr_{1,3}} & & \\
	 		& \Zhat^d\times \T^d \ar[dl]_{\pr_1} \ar[dr]_{\pr_2}& & \T^d\times \Zhat^d \ar[dl]^{\pr_1} \ar[dr]^{\pr_2} & \\
	 		\Zhat^d & & \T^d & & \Zhat^d,
 		}
	 	\end{equation*}
	 from which one deduces that 
	 	\[[\overline{\mathcal{F}}^\times_d]\hotimes_{\T^d}[\FM_d] = (-1)^d[\Zhat^d\corl{\pr_2}(\T^d\times \Zhat^d\times \Zhat^d, \pr_{2,1}^*(\overline{\mathcal{P}}_d)\hotimes \pr_{1,3}^*(\overline{\mathcal{P}}_d)) \corr{\pr_3} \Zhat^d]. \]
	 We note that the bundle $\pr_{2,1}^*(\overline{\mathcal{P}}_d)\hotimes \pr_{1,3}^*(\overline{\mathcal{P}}_d)$ over $\T^d\times \Zhat^{2d}$ can be identified with trivial line bundle $(\R^d\times \Zhat^d\times \Zhat^d) \times \C$ modulo the relation 
	 	\[(x,\chi, \eta, \lambda) \sim (x+n,\chi, \eta, (\overline{\chi}\otimes\eta)(n)\lambda) \]
	 and from this observe that 
	 	\[{\rm PD}^{-1}_{spin}([\overline{\mathcal{F}}^\times_d]\hotimes_{\T^d}[\F_d]) = [(\T^d,\overline{\Delta})]_!\hotimes_{\T^{2d}}[\mathcal{F}_{2d}], \]
	 $\overline{\Delta}:\T^d\to \T^{2d}$ is the map defined in Example~\ref{example:invert torus}. Using the result from Example~\ref{example:invert torus} it follows that 
	 	\[{\rm PD}^{-1}_{spin}([\overline{\mathcal{F}}^\times_d]\hotimes_{\T^d}[\F_d]) = (-1)^{d+\frac{d(d-1)}{2}}[(\Zhat^d,\Delta)]_! \]
	 and so by Poincar\'{e} duality we have that 
	 	\[[\overline{\mathcal{F}}_d]\hotimes_{\T^d}[\F_d] = (-1)^{\frac{d(d-1)}{2}}\cdot[\overline{\mathcal{F}}^\times_d]\hotimes_{\T^d}[\F_d] = [\Zhat^d\corl{\id} (\Zhat^d,\mathbf{1}) \corr{\id} \Zhat^d] =  \mathbf{1}\in \KK_0(\Zhat^d,\Zhat^d). \]
	 The equation $[\F_d]\hotimes_{\Zhat^d}[\overline{\mathcal{F}}_d] = \mathbf{1} \in \KK_0(\T^d,\T^d)$ follows in exactly the same way, whence  we get the result.
\end{proof}

\section{A Geometric Description of Baum-Connes Assembly}
\label{section:Baum-Connes}
Theorem~\ref{thm: MainK-theoryComputation} leads to nice, geometric interpretation of the Baum-Connes assembly map \cite{BCH} for free abelian groups. We start by recalling the assembly map, and then the main result. 

The Baum-Connes assembly map 
	\[\mu \colon \K_*(BG) \to \K_*(C^*G)\]
for discrete groups with finite \(BG\) is defined by 
	\[\mu (f) =  \poincare_G \otimes_{C(BG)\otimes C^*G} (f\otimes 1_{C^*G})\]
where \(\poincare \in \KK_0(\C, C(BG)\otimes C^*G)\) is the class of the Mischenko element, 
the class of a finitely generated projective 
\(C(BG)\otimes C^*(G)\)-module reducing to the class of the Poincar\'e bundle \(\poincare_d\) if \(G = \Z^d\).

The Baum-Connes assembly map is a higher-index theory construction. If \(G = \Z^d\), \(BG = \T^d\), then Fourier transform identifies Baum-Connes with a map 
	\[\mu \colon \KK_*(C(\T^d), \C) \to \KK_*(\C, C(\That^d)).\]
The following observations follow from the definitions. 

1. If \(D\) is the Dirac operator on \(\T^d\) acting on \(L^2(\T^d, S)\) for an appropriate spinor bundle, then \(D\otimes 1\) acts on the Hilbert \(C(\That^d)\)-module \(L^2(S)\otimes C(\That^d)\) and can then be twisted by the Poincar\'e bundle. This results in the bundle of Dirac operators \(\{ D_\chi\}_{\chi \in \Zhat^d}\) over \(\Zhat^d\) whose families index is \(\mu ([D])\). 

2. If \(i \colon T \to \T^d\) is an oriented subtorus, the Poincar\'e bundle may be pulled back along \(i\) and the same procedure produces a bundle of Dirac operators (on \(T\)) over \(\That^d\). The families index is \(\mu ( i_*([D])\). 

The Baum-Connes assembly map is a special case of a \emph{duality} map in KK. Recall that if \(X\) is a locally compact space, then a \emph{KK-dual} for \(X\) of dimension \(d\) is a locally compact \(X'\) together with a pair of classes 
	\[\widehat{\Delta}\in \KK_d(\pnt, X\times X'), \;\;\;\;\; \Delta \in \KK_{d} (X\times X', \pnt)\]
called the unit and co-unit respectively, which satisfy the zig-zag equations 

	\begin{equation}
		(\widehat{\Delta}\hotimes 1_X) \, \hotimes_{X\times X'\times X} \, (1_X\hotimes \sigma^*\Delta) =1_X\;\;\; \; \; \sigma_*(\widehat{\Delta}\hotimes 1_{X'}) \, \hotimes_{X'\times X\times X'} \, (1_{X'} \hotimes\sigma^* \Delta) =1_{X'}.
	\end{equation}

The class \(\widehat{\Delta}\) induces a map 
	\begin{multline}
	\label{equation:PDspin}
		\PD   \colon \KK_*(X\times U, V) \to \KK_{*+d} (U, X\times V),\\
		\PD (f) :=   \widehat{\Delta} \otimes_{X\times X} (1_X\times f),
	\end{multline}
using the notation of \eqref{align:cupcapproduct}, and the class \(\Delta\) induces a map inverting it. In particular, if \(X\) and \(X'\) are dual then the K-theory of \(X\) is isomorphic to the K-homology of \(X'\), and the K-theory of \(X'\) is isomorphic to the K-homology of \(X\). 

The key example in index theory and differential topology is \emph{spin duality}, which we have 
already discussed.  The unit and co-unit are classes of the correspondences  
	\[ \Dudeltaspin := \left[ X\xrightarrow{\delta} X\times X\right],\;\;\; \Deltaspin := \left[ X\times X \xleftarrow{\delta} X \right]\]
where \(\delta\) is the diagonal map, and the maps are K-oriented by the assumed K-orientation on \(X\). In this case, \(\sigma_*(\Delta) = \Delta\), and \( \sigma^*(\widehat{\Delta}) = (-1)^d \; \widehat{\Delta}\) it follows that these two equations reduce to a single one, which is easily check by hand, using composition of correspondences by transversality. We have already noted that the resulting duality 
maps flip the legs on a correspondence.

It is possible for a given \(X\) to have than one dual in KK. But any two duals are intertwined by a KK-equivalence. 

	\begin{lemma}
	\label{lemma:ondualities1}
		Suppose \(X\) is a smooth manifold, with two duals \(X'\) and \(X''\) in \(\KK\), of dimensions \(k'\) and \(k''\). Then there is a unique \(\KK\)-equivalence 
			\[\psi \in \KK_{k'-k''} (X', X'')\]
		making the diagram 
			\begin{equation}
			\xymatrix{ 
				\KK_*(X\times U, V)\ar[dr]^{\PD''} \ar[r]^{\PD'}
				& \KK_{*+k'} (U, X'\times V) \ar[d]^{\hotimes_{X'} f}\\
				& \KK_{*+k''} (U, X''\times V) } 
			\end{equation}
commute. Moreover, \( \psi = (\PD')^{-1} (\Delta'')\). 

If  \(\psi\colon \KK_d(X', X'')\) is a \(\KK\)-equivalence, where \(X'\) is a dual for \(X\) with unit \(\Delta'\) and co-unit \(\widehat{\Delta}'\), then \(X''\) is a dual for \(X\) with co-unit \( \Delta'' := (1_X\otimes \psi ^{-1}) \otimes_{X\times X'} \Delta'\) and unit \( \widehat{\Delta}'' := \widehat{\Delta}' \otimes_{X\times X'} (1_X\otimes \psi).\)
\end{lemma}

The proof is routine. 


	\begin{theorem}
	\label{theorem:fmduality}
		Define 
			\begin{align}
				\Dudeltafm := [\poincare_d]\in \KK_0(\pnt\; , \T^d\times \That^d),\\
				\Deltafm :=[\dol_d \cdot \poincare_d] \in \KK_0(C(\T^d\times \Zhat^d) ,  \C).
			\end{align}
		where \(\dol_d \cdot \poincare_d\) is the Dirac-Dolbeault operator \(\partial_d\) on \(\T^d\times \That^d\), twisted by the Poincar\'e line bundle. Then \(\Deltafm\) and \(\Dudeltafm\) are the unit and co-unit (respectively) of a \(0\)-dimensional \(\KK\)-duality between \(\T^d\) and \(\That^d\). The associated duality map for any \(A\),  
			\[ \PD_{\mathrm{FM}} \colon \KK_*(C(\T^d), A) \to \KK_0(\C, A\otimes C^*(\Z^d)), \;\;\; f\mapsto [\poincare_d]\otimes_{C(\T^d)\otimes C^*(\Z^d)} (f\otimes 1_{C^*(\Z^d)}) \]
		is the Baum-Connes assembly map. 
	\end{theorem}

We will refer to the duality of Theorem \ref{theorem:fmduality} as \emph{Baum-Connes duality}. 

	\begin{proof}[Proof of Theorem~\ref{theorem:fmduality}]
		We apply the second part of the Lemma with \(f = \FM_d\), \(X'= \T^d, X''= \That^d\). Let \(\Delta' = \Deltaspin\) and \(\widehat{\Delta} ' = \Dudeltaspin\), then the unit class \(\Dudeltaspin \otimes_{\T^d\times \T^d} (1_{\T^d}\otimes \FM_d) \) for the new duality is by the definitions equal to \(\PDspin (\FM_d)\). Since spin duality flips the legs of a correspondence this equals \( [\poincare_d]\in \KK_0(\pnt, \T^d\times \That^d)\). Thus, \(\Deltafm = [\poincare_d]\). 
 
		Similar computations show that the new co-unit is as stated using the Fourier-Mukai inversion formula. The Baum-Connes assembly map by definition is the map 
			\[ f \mapsto [\poincare_d]\otimes_{\T^d\times \That^d} (f\otimes 1_{\That^d}).\]
		and this is by definition equal to the map \(\PDfm\) of the new duality. 
	\end{proof}

We obtain the following factorization of the Baum-Connes assembly map. 

	\begin{lemma}
	\label{lemma:factoringbc}
			\[\xymatrix{  \K_*(\T^d)  \ar[r]^{\PDspin \;\;\;\;} \ar[dr]_{\mu} & \K^{*+d} (\T^d) \ar[d]^{(\FM_d)_*} \\  
		& \K^*(\That^d) }.\]
		where the map \((\FM_d)_*\) is composition with the Fourier-Mukai transform. 
	\end{lemma}

	\begin{proof}
		Let \(f \in \KK_*(\T^d, \pnt)\) then 
			\begin{multline*}
				(\FM_d)_*\PDspin (f)  = \left( \Dudeltaspin \otimes_{\T^d\times \T^d} (f\otimes 1_{\T^d}) \right) \otimes_{\T^d} [\FM_d] = \Dudeltaspin \otimes_{\T^d\times \T^d} \left( (f\otimes 1_{\T^d}) \otimes_{\T^d} [\FM_d] \right) \\
				= \Dudeltaspin \otimes_{\T^d\times \T^d} (f\otimes_\C [\FM_d]) = \Dudeltaspin \otimes_{\T^d\times \T^d} ([\FM_d]\otimes_\C f) \\
				= \Dudeltaspin \otimes_{\T^d\times \T^d} \left( ([\FM_d]\otimes 1_{\T^d}) \otimes_{\That^d \times \T^d} ( 1_{\That^d} \otimes f) \right)  = (\FM_d\otimes 1_{\T^d})_*(\Dudeltaspin) \otimes_{\That^d\times \T^d} (1_{\That^d} \otimes f) \\ = \PDfm (f) = \mu (f).
 			\end{multline*}
	\end{proof}

The Lemma in combination with Theorem \ref{thm: MainK-HomologyComputation}
 gives the main result of this note. 

	\begin{corollary}
		Let \(T \subset \T^d\) be a K-oriented subtorus of dimension \(j\) defining the element \([(T,i)]_* \in \K_{j} (\T^d)\). Then 
			\[ \mu ([(T,i)]_*) =  (-1)^{k(d-k)+\frac{k(k-1)}{2}}\cdot [\widehat{(T,i)}]! \in \K^{-j} (\That^d),\]
		where \(\mu\) is the Baum-Connes assembly map. 
	\end{corollary}

As noted above, this supplies a topological description of the higher analytic index \(\mu ([T]_*)\), and so might be regarded as a higher index theorem. 

\section{Further Remarks}

Baum-Connes duality can be of course formulated for discrete groups \(\Gamma\) with finite \(B\Gamma\) with unit the finitely generated projective Mischenko-Poincar\'e module \(\poincare_\Gamma\) over \(C(G\Gamma) \otimes C^*(\Gamma)\). Constructing a \emph{candidate} for the 
unit is tantamount to applying the Dirac-dual-Dirac method in a geometric 
setting; we describe such a candidate below in the case of \(\Z^n\). But the 
zig-zag equations are a statement of the kind \(\gamma = 1\) for the group in question 
and are almost 
certainly not satisfied even by all lattices in Lie groups. For this reason, this 
kind of duality has not been studied much. Moreover, 
the equivariant version of the Baum-Connes map 
 allows coefficients in a \(\Gamma\)-C*-algebra, whereas the duality version 
 proposed above does not.

Recently, however, \cite{NPW}) studies 
Baum-Connes duality for \(\Gamma = \Z^d\)), which, as they show, holds 
equivariantly with respect to a finite group action. This leads to a 
certain relationship with Langland's duality. 

 In this article, we have 
described Baum-Connes duality in terms of correspondences and the Fourier-Mukai transform. 
The uniqueness of a co-unit in a duality with given unit implies that the 
class \(\Deltafm\) must agree with the co-unit of \cite{NPW}, which admits a quite 
different, analytic description in terms of Dirac-Schr\"odinger operators, which we now describe.

The standard construction of a \(\gamma\)-element for fundamental groups \(\Gamma = \pi_1 (M)\) of nonpositively curved manifolds is based on analysis and some geometry of \(\tilde{M}\) as follows. Let \(H\) denote the Hilbert space of \(L^2\)-differential forms on \(\tilde{M}\). It carries a natural unitary action of \(\Gamma\), and a representation of \(C(M)\) by lifting functions on \(M\) to periodic functions on \(\tilde{M}\) and then multiplication operators. Thus, \(H\) carries a representation of \(C(M)\otimes C^*(\Gamma)\). 

Let \(\derham = d+d^*\) be the de Rham operator, densely defined unbounded  operator on \(H\) with a canonical self-adjoint extension. Fix a basepoint \(o \in \tilde{M}\) and let \(X\) denote the operator of interior multiplication by the co-vector field \(df\) on \(\tilde{M}\) where \(f(x) = \rho(x,o)^2\), \(\rho\) the Riemannian distance. Then the operator 
	\[\derham + X\] 
commutes modulo bounded operators with \(C^\infty (M)\otimes \C[\Gamma]\) and determines a spectral (unbounded) cycle \( (H, \pi, \derham + X)\) for \(\KK_0( C(M)\otimes C^*(\Gamma), \C)\). The reasons are:

1. \(\derham + X\) is Fredholm, 

2. \(C^\infty(M)\) acts by periodic functions, all of whose derivatives are uniformly bounded on \(\tilde{M}\),
making the commutators \([f, \derham+X] = [f, \derham] = df\) bounded, and

 3. The \(\Gamma\) action is isometric on \(\tilde{M}\), and the non-positive curvature assumption implies that \(\gamma (X) - X\) is bounded, for all \(\gamma\in \Gamma\), because the \(\gamma\)-action moves the basepoint to \(\gamma (o)\), and the difference is bounded due to thin geodesic triangles in a nonpositively curved space. 

Specializing to \(\Gamma = \Z^d\) gives 

	\begin{theorem}
	\label{theorem:storyoftwocounits} (see \cite{NPW}.) 
		The class \(\Deltads\in \KK_0( C(\T^d)\otimes C^*(\Z^d), \C)\) of the spectral cycle 
			\[ \left( L^2\left(T^*(\R^n) \right), \pi, \derham + X\right)\] is
		the co-unit for a duality between \(C^*(\Z^d)\) and \(C(\T^d)\) with unit the class \([\poincare_d]\in \KK_0(\C, C(\T^d)\otimes C^*(\Z^d))\) of the Poincar\'e bundle. In particular, \(\Deltads = \Deltafm\). 
	\end{theorem}

That is, the class of the Dirac-Schr\"odinger cycle defined above agrees with that of the spin dual of the Fourier-Mukai transform. 

However, the equality \(\Deltads = \Deltafm\) can be proved directly using an extremely natural unitary equivalence of Hilbert spaces. We sketch the argument.

Let \(\Gamma (\poincare_d)\) be the right Hilbert \(C(\T)\otimes C^*(\Z) )\)-module of sections of the Poincar\'e bundle,
	\[ \Gamma ( \poincare_d)  := \{ f \in C\left(\R^d, C^*(\Z^d)\right) \; | \; f(x+n) = f(x) \, [n] \in \Z^d,\;  x\in \R^d\},\]
where \([n]\in C^*(\Z^d)\) is the unitary corresponding to the integer \(n\). The right \(C(\T^d)\otimes C^*(\Z^d)\)-module structure is given by 
	\[ (f\cdot \phi ) (x) = f(x) \phi (x) , \;\;\;( f\cdot [n]) (x) = f(x) \cdot [n].\]
If we set \( \langle f_1, f_2\rangle  (x)  = f_1(x)^*f_2(x)\in C^*(\Z^d)\) then \(\langle f_1, f_2\rangle\) is periodic and so is an element of \(C(\T^d)\otimes C^*(\Z^d)\).

Let \(L^2(\poincare_d)\) denote the tensor product 
	\[ \Gamma (\poincare)\otimes_{C(\T^d)\otimes C^*(\Z^d)} L^2(\T^d)\otimes C^*(\Z^d)\]
of Hilbert modules over the representation \( \pi \colon C(\T^d)\otimes C^*(\Z^d) \to \Bound \left( L^2(\T^d)\otimes l^2(\Z^d)\right)\) of multiplication operators and group translation operators. Equivalently, \(L^2(\poincare)\) is the Hilbert module completion to a Hilbert space using the trace \(\tau \colon C(\T^d)\otimes C^*(\Z^d) \to \C, \;\; \tau (\sum f_n [n]) := \int_{\T^d} f_0.\)

	\begin{lemma}
		\(L^2(\poincare_d) \cong L^2(\R^d)\) under a canonical unitary equivalence. Under this identification, the representation \(\pi\) corresponds to the integer  action of \(C^*(\Z^d)\) on \(L^2(\R^d)\) induced by the translation action of \(\Z^d\) on \(\R^d\), and the action of \(C(\T^d)\) by multiplication by \(\Z^d\)-periodic functions. 
	\end{lemma}
	\begin{proof}
	If \(f\in \poincare_d\) then we expand \(f\) pointwise in its Fourier series to write \(f(x) = \sum_{n\in \Z^d}  f_n(x)  [n]\in C^*(\Z^d) \) for a family \(\{f_n\}_{n \in \Z^d}\). The condition to be in \(\poincare\) gives that \(f_{n+1} (x) = f_n(x+1)\). Hence \(f_n = n(f_0)\). So \(f\) is completely determined by its zero coefficient in this expansion. If \(\rho \in C_c(\R)\) let 
		\[\hat{\rho } (x) = \sum_{n \in \Z^d} \rho (x-n) \cdot [n].\]
	Then \(\hat{\rho}\in \Gamma(\poincare_d)\), and we have 
		\[ \hat{\rho} (x)^* \hat{\rho} (x) = \sum_{n, m\in \Z^d} \overline{\rho (x-n) }) \rho(x-m)  \cdot [n-m] .\] 
	Let \(\xi_0  = 1\otimes \delta_0 \in L^2(\T^d)\otimes l^2(\Z^d)\) and define 
		\[ U \colon C_c(\R) \to L^2(\poincare),\;\;\; U\rho = \hat{\rho} \otimes \xi_0.\] 
	Then 
		\begin{equation}
			\langle U\rho, U\rho\rangle = \langle \xi_0, \hat{\rho}^*\hat{\rho} \cdot \xi_0 \rangle = \sum_{n t\in \Z^d} \int_{\T^d} \abs{\rho (x-n)}^2 dx = \norm{\rho}^2_{L^2(\R^d)}.
		\end{equation}
\end{proof}

By Theorem \ref{theorem:fmduality}, Baum-Connes duality has co-unit 
\([\dol \cdot \poincare]\). On the other hand, the operation of  
twisting an elliptic operator \(D\) on a compact 
manifold \(X\) by a vector bundle \(E \to X\) may be expressed in 
\(\KK\)-theoretic terms as follows: 
the class \([D_E]\) of the twisted operator satisfies the equation 
\[ [D_E] = \delta^*( [E]\otimes 1_{C(X)}) \otimes_{C(X)} [D]\]
where \(\delta \colon X \to X\times X\) is the diagonal map.  Applying this to 
\(D = \dol\) and \(E = \poincare\) and computing the corresponding 
Hilbert module composition results by the above Lemma in 
a unitary equivalence with the Hilbert space 
\(H = L^2(\Lambda^*(\R^n))\). 

The axioms for a Kasparov product imply that 
the Dirac-Schr\"odinger operator \(\derham + X\), acting on \(H = 
L^2(\R^n, \Lambda^*\R^n)\) represents \([\dol \cdot \poincare]\). 

On of the interests in this is that the Dirac-Schr\"odinger cycle can be 
`quantized' to give a K-homology cycle for \(C(\T^d)\rtimes \Gamma\), for 
any lattice in \(\R^n\), not just the lattice \(C^*(\Z^d)\). This classes 
and associated index formulae 
are studied in dimension \(1\) in the paper \cite{E:Kids}. 

The paper 
\cite{Block} develops an analogue of the Fourer-Mukai 
transform for irrational tori; it is not clear to us if this work is 
related to ours.

\end{document}